\documentclass{amsart}

\usepackage{amsmath,amsfonts,amssymb}
\usepackage{tikz}

\newtheorem{theorem}{Theorem}
\newtheorem{lemma}[theorem]{Lemma}

\theoremstyle{definition}
\newtheorem{definition}{Definition}

\theoremstyle{remark}

\newcommand{\newword}[1]{\textbf{\emph{#1}}}

\newcommand{\Red}{\mathrm{Red}}
\newcommand{\KD}{\mathrm{KD}}

\newcommand{\schubert}{\ensuremath{\mathfrak{S}}}

\newcommand{\kohnert}{\ensuremath{\mathfrak{K}}}

\newcommand{\comp}[1]{\ensuremath\mathrm{\mathbf{#1}}}
\newcommand{\wt}{\comp{wt}}

\newcommand{\D}{\ensuremath\mathbb{D}}
\newcommand{\T}{\ensuremath\mathbb{T}}

\newcommand{\W}{\ensuremath\mathbb{W}}

\newlength\cellsize 

\savebox2{
\begin{picture}(8,8)
\put(0,0){\line(1,0){8}}
\put(0,0){\line(0,1){8}}
\put(8,0){\line(0,1){8}}
\put(0,8){\line(1,0){8}}
\end{picture}}

\newcommand\boxify[1]{\def\thearg{#1}\def\nothing{}%
\ifx\thearg\nothing\vrule width0pt height\cellsize depth0pt%
  \else\hbox to 0pt{\usebox2\hss}\fi%
  \vbox to \cellsize{\vss\hbox to \cellsize{\hss$_{#1}$\hss}\vss}}

\savebox3{
\begin{picture}(8,8)
\put(4,4){\circle{8}}
\end{picture}}

\newcommand{\circify}[1]{\def\thearg{#1}\def\nothing{}%
\ifx\thearg\nothing\vrule width0pt height\cellsize depth0pt%
  \else\hbox to 0pt{\usebox3\hss}\fi%
  \vbox to \cellsize{\vss\hbox to \cellsize{\hss$_{#1}$\hss}\vss}}

\newcommand\nullify[1]{\def\thearg{#1}\def\nothing{}%
\ifx\thearg\nothing\vrule width0pt height\cellsize depth0pt%
  \else\hbox to 0pt{\hss}\fi%
  \vbox to \cellsize{\vss\hbox to \cellsize{\hss$_{#1}$\hss}\vss}}

\newcommand\tableau[1]{\vtop{\let\\=\cr
\setlength\baselineskip{-8000pt}
\setlength\lineskiplimit{8000pt}
\setlength\lineskip{0pt}
\halign{&\boxify{##}\cr#1\crcr}}}

\newcommand\cirtab[1]{\vline\vtop{\let\\=\cr
\setlength\baselineskip{-8000pt}
\setlength\lineskiplimit{8000pt}
\setlength\lineskip{0pt}
\halign{&\circify{##}\cr#1\crcr}}}

\newcommand\nulltab[1]{\vtop{\let\\=\cr
\setlength\baselineskip{-8000pt}
\setlength\lineskiplimit{8000pt}
\setlength\lineskip{0pt}
\halign{&\nullify{##}\cr#1\crcr}}}

\newcommand{\cball}{%
  \begin{tikzpicture}
    \filldraw[fill=black!25,draw=black] circle (4pt);
  \end{tikzpicture}
}

\usepackage{xcolor}
\usepackage[colorinlistoftodos]{todonotes}

\begin{document}


\title[Kohnert's rule]{A bijective proof of Kohnert's rule \\ for Schubert polynomials}  

\author{Sami H. Assaf}
\address{University of Southern California, 3620 S. Vermont Ave., Los Angeles, CA 90089}
\email{shassaf@usc.edu}
\thanks{Work supported in part by NSF DMS-1763336.}

\subjclass[2010]{Primary: 05A05, 05A19; Secondary: 05E05, 14N10, 14N15}




\keywords{Schubert polynomials, Kohnert polynomials, reduced words}

\begin{abstract}
  Kohnert proposed the first monomial positive formula for Schubert polynomials as the generating polynomial for certain unit cell diagrams obtained from the Rothe diagram of a permutation. Billey, Jockusch and Stanley gave the first proven formula for Schubert polynomials as the generating polynomial for compatible sequences of reduced words of a permutation. In this paper, we give an explicit bijection between these two models, thereby definitively proving Kohnert's rule for Schubert polynomials.
\end{abstract}

\maketitle


Schubert polynomials, introduced by Lascoux and Sch{\"u}tzenberger in 1982 \cite{LS82}, are polynomial representatives of Schubert classes for the cohomology of the flag manifold whose structure constants precisely give the Schubert cell decomposition for the corresponding product of Schubert classes. They can be defined in terms of \newword{divided difference operators} $\partial_i$ acting on polynomials $f \in \mathbb{Z}[x_1,x_2,\ldots]$ by
\begin{equation}
  \partial_i(f) = \frac{f - s_i \cdot f}{x_i - x_{i+1}} ,
\end{equation}
where $s_i \in \mathcal{S}_{\infty}$ is the simple transposition acting by exchanging $x_i$ and $x_{i+1}$.

A \newword{reduced word} for a permutation $w$ is a sequence $\rho = (\rho_{\ell},\ldots,\rho_1)$ such that $s_{\rho_{\ell}} \cdots s_{\rho_1} = w$ has $\ell$ inversions. Let $\Red(w)$ denote the set of reduced words for $w$. One can show the $\partial_i$ satisfy the braid relations, and so for $\rho\in\Red(w)$ we define
\begin{eqnarray}
  \partial_w & = & \partial_{\rho_{\ell}} \cdots \partial_{\rho_1} .
\end{eqnarray}

\begin{definition}[\cite{LS82}]
  The \newword{Schubert polynomial} $\schubert_w$ for $w\in\mathcal{S}_n$ is
  \begin{equation}
    \schubert_w = \partial_{w^{-1} w_0} \left( x_1^{n-1} x_2^{n-2} \cdots x_{n-1} \right),
    \label{e:schubert}
  \end{equation}
  where $w_0 = n \cdots 2 1$ is the longest permutation of $\mathcal{S}_n$ of length $\binom{n}{2}$.
  \label{def:schubert}
\end{definition}

Surprisingly, at least from this definition, Schubert polynomials are monomial positive, and so one naturally seeks a manifestly positive combinatorial expression in terms of monomials. Kohnert, a student of Lascoux, proposed the first such formula in 1990 using combinatorics of \newword{diagrams}, finite collections of points, called cells, in the first quadrant of $\mathbb{Z}\times\mathbb{Z}$. Kohnert's elegant combinatorial model for Demazure characters \cite{Koh91} uses the following operation.

\begin{definition}[\cite{Koh91}]
  A \newword{Kohnert move} on a diagram selects the rightmost cell of a given row and moves the cell down within its column to the first available position below, if it exists, jumping over other cells in its way as needed. 
  \label{def:kohnert_move}
\end{definition}

Let $\KD(w)$ denote the set of \newword{Kohnert diagrams for $w$}, defined as those diagrams obtainable from the \newword{Rothe diagram} $\D(w)$ via Kohnert moves, where
\begin{equation}
  \D(w) = \{ (i,w_j) \mid i<j \mbox{ and } w_i > w_j \} .
  \label{e:rothe}
\end{equation}
For example, Fig.~\ref{fig:KMs} shows the Rothe diagram for $w=152869347$ on the left, and all other diagrams are obtained from Kohnert moves on the shaded cells.

\begin{figure}[ht]
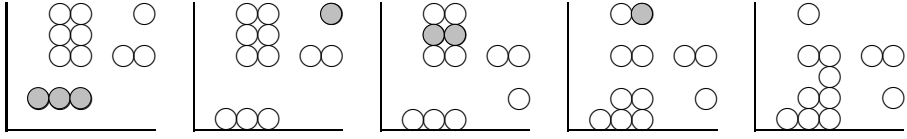

  \begin{displaymath}
    \arraycolsep=0.9\cellsize
    \begin{array}{ccccc}
      \cirtab{ 
        & & \ & \ & & & \ \\
        & & \ & \ \\
        & & \ & \ & & \ & \ \\
        & & \\
        & \cball & \cball & \cball \\ & \\\hline } &
      \cirtab{ 
        & & \ & \ & & & \cball \\
        & & \ & \ \\
        & & \ & \ & & \ & \ \\
        & & \\
        \\
        & \ & \ & \ & \\\hline } &
      \cirtab{ 
        & & \ & \ & & & \\
        & & \cball & \cball \\
        & & \ & \ & & \ & \ \\
        & & \\
        & & & & & & \ \\
        & \ & \ & \  \\\hline } &
      \cirtab{ 
        & & \ & \cball & & & \\
        & & \\
        & & \ & \ & & \ & \ \\
        \\
        & & \ & \ & & & \ \\
        & \ & \ & \ &  \\\hline } &
      \cirtab{ 
        & & \ & & & & \\
        & & \\
        & & \ & \ & & \ & \ \\
        & & & \ \\
        & & \ & \ & & & \ \\
        & \ & \ & \ \\\hline }
    \end{array}
  \end{displaymath}
  \caption{\label{fig:KMs}Several Kohnert diagrams for $152869347$.}
\end{figure}

\begin{definition}
  The \newword{Kohnert polynomial} for a permutation $w$ is 
  \begin{equation}
    \kohnert_{w} = \sum_{T \in \KD(w)} x_1^{\wt(T)_1} \cdots x_n^{\wt(T)_n}.
    \label{e:kohnert}
  \end{equation}
  where the \newword{weight of $T$} satisfies $\wt(T)_i$ equals the number of cells of $T$ in row $i$.
  \label{def:kohnert_poly}
\end{definition}

Kohnert asserted $\schubert_w = \kohnert_w$, giving the first conjectured combinatorial formula for Schubert polynomials. Winkel \cite{Win99,Win02} gives two proofs, the second an improvement on the first after it faced broad criticism, though neither the original nor revised proof is widely accepted given the intricate and opaque arguments. 

Billey, Jockusch, and Stanley \cite{BJS93} gave the first proven combinatorial formula for Schubert polynomials in terms of compatible sequences for reduced words. For $\rho \in \Red(w)$, a word $\alpha$ is \newword{$\rho$-compatible} if $\alpha_j \leq \rho_j$, $\alpha_{j+1} \geq \alpha_j$, and $\alpha_{j+1} > \alpha_{j}$ whenever $\rho_{j+1} > \rho_{j}$. In this case, we call $(\rho,\alpha)$ a \newword{compatible pair}.

\begin{figure}[ht]
  \begin{displaymath}
    \arraycolsep=2pt
    \begin{array}{ccccccccccccccccccccccc}
      6 & & 4 & & 5 & & 7 & & 8 & & 6 & & 4 & & 5 & & 7 & & 2 & & 3 & & 4 \\
      \rotatebox[origin=c]{90}{$\leq$} & &
      \rotatebox[origin=c]{90}{$\leq$} & &
      \rotatebox[origin=c]{90}{$\leq$} & &
      \rotatebox[origin=c]{90}{$\leq$} & &
      \rotatebox[origin=c]{90}{$\leq$} & &
      \rotatebox[origin=c]{90}{$\leq$} & &
      \rotatebox[origin=c]{90}{$\leq$} & &
      \rotatebox[origin=c]{90}{$\leq$} & &
      \rotatebox[origin=c]{90}{$\leq$} & &
      \rotatebox[origin=c]{90}{$\leq$} & &
      \rotatebox[origin=c]{90}{$\leq$} & &
      \rotatebox[origin=c]{90}{$\leq$} \\
      \alpha_{12} & > &
      \alpha_{11} & \geq &
      \alpha_{10} & \geq &
      \alpha_{9}  & \geq &
      \alpha_{8}  & > &
      \alpha_{7}  & > &
      \alpha_{6}  & \geq &
      \alpha_{5}  & \geq &
      \alpha_{4}  & > &
      \alpha_{3}  & \geq &
      \alpha_{2}  & \geq &
      \alpha_{1} 
    \end{array}
  \end{displaymath}
  \caption{\label{fig:comp}Equations for $\rho$-compatible words.}
\end{figure}

For example, Fig.~\ref{fig:comp} shows the necessary inequalities for a compatible pair for $\rho=(6,4,5,7,8,6,4,5,7,2,3,4)$, giving two solutions $(6,4,4,4,4,3,2,2,2,1,1,1)$ and $(5,4,4,4,4,3,2,2,2,1,1,1)$. By \cite[Theorem~1.1]{BJS93}, we have the following.

\begin{theorem}[\cite{BJS93}]
  The Schubert polynomial $\schubert_w$ is given by
  \begin{equation}
    \schubert_w = \sum_{\substack{\rho \in R(w) \\ \alpha \ \rho-\mathrm{compatible}}} x_{\alpha_1} \cdots x_{\alpha_{\ell}} . 
    \label{e:BJS}
  \end{equation}
  \label{thm:BJS}
\end{theorem}

We prove \eqref{e:kohnert} and \eqref{e:BJS} agree term by term by giving weight-preserving bijections
\[
    \left\{ (\rho,\alpha) \mid \rho\in\Red(w), \ \alpha \ \rho-\mathrm{compatible} \right\}
    \begin{array}{c}
      \stackrel{\D}{\longrightarrow} \\[-1ex]
      \stackrel{\displaystyle\longleftarrow}{\scriptstyle\W}
    \end{array}
    \KD(w) .
\]
These recursively defined maps are easily implemented, proving the following.

\begin{theorem}[Kohnert's rule]
  Given a permutation $w$, we have
  \begin{equation}
    \kohnert_w = \schubert_w .
  \end{equation}
\end{theorem}


As motivation, consider the canonical word associated with $\D(w)$ \cite[Def.~2.3]{Ass19-inv}.

\begin{definition}[\cite{Ass19-inv}]
  A word $\pi$ is \newword{super-Yamanouchi} if factors into increasing words $(\pi^{(k)} | \cdots | \pi^{(1)})$ with each $\pi^{(i)}$ an interval and $\min(\pi^{(i+1)}) > \min(\pi^{(i)})$.
  \label{def:superY}
\end{definition}

By \cite[Prop.~2.4]{Ass19-inv}, each permutation $w$ has a unique super-Yamanouchi reduced word which we denote by $\pi_w$. By \cite[Prop.~3.2]{Ass19-inv}, $\pi_w$ is obtained by filling cells in row $r$ of $\D(w)$ with entries $r,r+1,\ldots$ from left to right. By \cite[Prop.~3.3]{Ass19-inv}, column $c$ of $\D(w)$ so filled has entries $c,c+1,\ldots$ from bottom to top.

For example, the super-Yamanouchi reduced word for the permutation $152869347$ is $\pi_w = (6,7,8\mid 5,6\mid 4,5,6,7\mid 2,3,4)$. Fig.~\ref{fig:superY} shows $\D(w)$ labeled by $\pi_w$. Notice the rows of letters of $\pi_w$ correspond to values of the maximal compatible sequence.

\begin{figure}[ht]
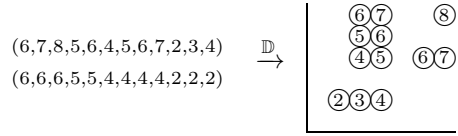

  \begin{displaymath}
    \arraycolsep=0.5\cellsize
    \begin{array}{ccc}
      \raisebox{-2\cellsize}{$
      \begin{array}{c}
        \scriptstyle (6,7,8,5,6,4,5,6,7,2,3,4) \\
        \scriptstyle (6,6,6,5,5,4,4,4,4,2,2,2)
      \end{array}$} &
      \raisebox{-2\cellsize}{$\xrightarrow{\D}$} &
      \cirtab{ 
        & & 6 & 7 & & & 8 \\
        & & 5 & 6 \\
        & & 4 & 5 & & 6 & 7 \\
        & & \\
        & 2 & 3 & 4 \\ & \\\hline }
    \end{array}
  \end{displaymath}
  \caption{\label{fig:superY}The super-Yamanouchi word and Rothe diagram for $152869347$.}
\end{figure}

We use super-Yamanouchi words to relate Rothe diagrams for permutations related in Bruhat order, thus motivating the action of permutations on diagrams. 

\begin{lemma}
  Let $\pi = \pi_w$. If $i=1$ or $(\pi_w)_i > (\pi_w)_{i-1}$, then $\widehat{\pi} = \pi_{\ell}\cdots\pi_{i+1}\pi_{i-1}\cdots\pi_1$ is super-Yamanouchi. Moreover, if $\pi_i$ lies in column $c$ of $\D(w)$, then no lower row of $\D(w)$ has its rightmost cell in column $c$ if and only if $\widehat{\pi} = \pi_{\widehat{w}}$ where $\widehat{w} = w s_{c}$, and in this case $\D(\widehat{w})$ is $\D(w) \setminus \{(\pi_w)_i\}$ with columns $c,c+1$ permuted.
  \label{lem:superY}
\end{lemma}

\begin{proof}
  If $i=1$, then $\widehat{\pi}$ must still factor into intervals, with the last now one letter shorter (possibly deleted), but the smallest of each interval is unchanged, and so $\widehat{\pi}$ is super-Yamanouchi. If $i=\ell$, then $\widehat{\pi}$ again factors into intervals, where now we have deleted the leftmost, so  $\widehat{\pi}$ remains super-Yamanouchi. If $\ell>i>1$, then either  $\pi_{i+1} = \pi_i + 1$, and so we have removed an interval, or $\pi_{i+1}=\pi_i - 1$ lies in the same interval as $\pi_i$, so the smallest letter in that interval (and hence $\pi_{i+1}$ as well) must be larger than $\pi_{i-1}$, thus ensuring the super-Yamanouchi criteria are met.
  
  Since $\pi_i > \pi_{i-1}$ (or $i=1$), $\pi_i$ must lie at the end of its row, say $r$, in $\D(w)$. Since rows and columns of $\D(w)$ are intervals, if the nearest occupied row below $r$ is shorter than row $r$, then all letters in that row are less than $\pi_i-1$, and so $\pi_i$ commutes with them. If the row is strictly longer, then it contains $\pi_{i-1}$ followed by $\pi_i$, and so $\pi_i$ may braid through this row and become $\pi_{i}-1$, at which point it commutes with any remaining letters in the row, all of which are at least $\pi_{i}+1$. The same logic now continues down, until at last we have transformed $\pi$ into the word $\widehat{\pi}c$, where $c$ is the last value through which $\pi_i$ braided, which by \cite[Prop.~3.3]{Ass19-inv} is the column index in which it sits. Thus $\widehat{\pi} = \pi_{w s_{c}}$.

  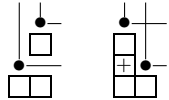
\begin{figure}[ht]
    \begin{center}
      \begin{tikzpicture}[every node/.style={inner sep=0pt},xscale=0.28,yscale=0.28]
        \node at (1,1)  {$\boxify{\ }$};
        \node at (2,1)  {$\boxify{\ }$};
        \node at (2,3)  {$\boxify{\ }$};
        \node at (1,2) (AL) {$\bullet$};
        \node at (2,4) (AR) {$\bullet$};
        \draw  (AL) -- (1,5) ;
        \draw  (AL) -- (3,2) ;
        \draw  (AR) -- (2,5) ;
        \draw  (AR) -- (3,4) ;
        \node at (6,1)  {$\boxify{\ }$};
        \node at (7,1)  {$\boxify{\ }$};
        \node at (6,3)  {$\boxify{\ }$};
        \node at (6,2)  {$\boxify{+}$};
        \node at (6,4) (BL) {$\bullet$};
        \node at (7,2) (BR) {$\bullet$};
        \draw  (BL) -- (6,5) ;
        \draw  (BL) -- (8,4) ;
        \draw  (BR) -- (7,5) ;      
        \draw  (BR) -- (8,2) ;      
      \end{tikzpicture}
      \caption{\label{fig:Rothe}Going up in Bruhat order on Rothe diagrams.}
    \end{center}
  \end{figure}

  Since $\ell(w s_{c}) = \ell(w)-1$, we have $w_c > w_{c+1}$. In the Rothe diagrams, cells in rows $r < w_{c+1}$ are in either both or neither columns $c,c+1$, and no cells lie in either column weakly above row $w_{c}$. In rows $w_{c+1} < r < w_c$, $\D(w s_{c})$ has no cells in column $c$ while in $\D(w)$ has no cells in column $c+1$, and any cells in this range move between the two columns with $\D(w)$ having the additional cell in row $w_{c+1}$, column $c$ corresponding to the inversion $w_c > w_{c+1}$; see Fig.~\ref{fig:Rothe}.
\end{proof}

Lemma~\ref{lem:superY} relates the Rothe diagrams for $w$ and $w s_c$. To relate their Kohnert diagrams, we use \cite[Lemma~6.3.5]{Ass-KC} concerning \newword{southwest} diagrams, those such that for any two cells positioned with one strictly northwest of the other there exists a cell at their southwest corner. Rothe diagrams are easily seen to be southwest and, moreover, every pair of adjacent columns of $\D(w)$ is ordered by inclusion.

\begin{lemma}[\cite{Ass-KC}]
  Let $D$ be a southwest diagram, and let $c$ be a column index for which columns $c,c+1$ of $D$ are ordered by inclusion. Then $\mathfrak{s}_c D$, the diagram obtained from $D$ by permuting columns $c,c+1$, is southwest and there is a weight-preserving bijection between Kohnert diagrams for $D$ and for $\mathfrak{s}_c D$.
  \label{lem:vexillary}
\end{lemma}

The bijection asserted in Lemma~\ref{lem:vexillary} can be described as follows. Label cells in row $r$ of $D, \mathfrak{s}_c D$ with $r$. For any sequence of Kohnert moves on $D$, apply the moves while recording the entry of the cell that moves along with the rows between which it moves, and then apply the same sequence to $\mathfrak{s}_c D$ using the labels. 

\begin{definition}
  For $T$ a Kohnert diagram for $D$ and $c\geq 1$ a column index such that columns $c,c+1$ of $D$ are ordered by inclusion, let $\mathfrak{s}_c T$ be the corresponding Kohnert diagram for $\mathfrak{s}_c D$ under the bijection of Lemma~\ref{lem:vexillary}.
\end{definition}

We now state our bijections, though they are not obviously well-defined.

\begin{definition}
  For a compatible pair $(\rho,\alpha)$, the \newword{Kohnert diagram for $(\rho,\alpha)$} is
  \begin{equation}
    \D(\rho,\alpha) = x \cup \mathfrak{s}_{\rho_1} \D((\rho_{\ell},\ldots,\rho_2),(\alpha_{\ell},\ldots,\alpha_2)) ,
  \end{equation}
  where $x$ is the cell in row $\alpha_1$, column $\rho_1$ and $\D(\varnothing,\varnothing)= \varnothing$.
  \label{def:diagram}
\end{definition}

\begin{figure}[ht]
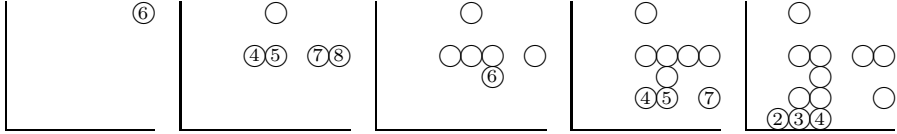

  \begin{displaymath}
    \arraycolsep=0.6\cellsize
    \begin{array}{ccccc}
      \cirtab{ & & & & & & 6 \\ \\ & \\ & \\ & \\ & \\\hline} &
      \cirtab{ & & & & \ \\ \\ & & & 4 & 5 & & 7 & 8 \\ & & \\ & & \\ & \\\hline} &
      \cirtab{ & & & & \ \\ \\ & & & \ & \ & \ & & \ \\ & & & & & 6 \\ & \\ & \\\hline} &
      \cirtab{ & & & \ \\ \\ & & & \ & \ & \ & \ \\ & & & & \ \\ & & & 4 & 5 & & 7 \\ & \\\hline} &
      \cirtab{ & & \ \\ \\ & & \ & \ & & \ & \ \\ & & & \ \\ & & \ & \ & & & \ \\ & 2 & 3 & 4  \\\hline} 
    \end{array}
  \end{displaymath}
  \caption{\label{fig:biject}The Kohnert diagram for the compatible pair $\rho=(6,4,5,7,8,6,4,5,7,2,3,4)$, $\alpha=(6,4,4,4,4,3,2,2,2,1,1,1)$.}
\end{figure}

Fig.~\ref{fig:biject} shows the Kohnert diagram for $(6\mid 4,5,7,8 \mid 6 \mid 4,5,7 \mid 2,3,4)$ with compatible word $(6,4,4,4,4,3,2,2,2,1,1,1)$. We begin with a cell in row $6$, column $6$; then we apply $\mathfrak{s}_8 \mathfrak{s}_7 \mathfrak{s}_5 \mathfrak{s}_4$ and add cells in row $4$, columns $4,5,7,8$; then we apply $\mathfrak{s}_6$ and add a cell in row $3$, column $6$; then we apply $\mathfrak{s}_7 \mathfrak{s}_5 \mathfrak{s}_4$ and add cells in row $2$, columns $4,5,7$; finally we apply $\mathfrak{s}_4 \mathfrak{s}_3 \mathfrak{s}_2$ and add cells in row $1$, columns $2,3,4$. 

The map from compatible pairs to diagrams is easily reversed as follows.

\begin{definition}
  For $T$ a nonempty diagram, the \newword{reduced word for $T$} is
  \begin{equation}
    \W(T) = \W(\mathfrak{s}_c (T\setminus x)) \ c
  \end{equation}
  where $x$ is the lowest then rightmost cell of $T$, and $c$ is its column index.
  \label{def:word}
\end{definition}

To show $\D$ is well-defined, we wish to match letters of $\rho\in\Red(w)$ with cells of $\D(w)$. As letters of $\pi_w$ naturally pair with cells of $\D(w)$, we use \cite[Def~2.7]{Ass19-inv} to define a matching from letters of $\pi_w$ (and hence cells of $\D(w)$) to letters of $\rho\in\Red(w)$.

\begin{definition}[\cite{Ass19-inv}]
  For $\rho,\pi\in\Red(w)$ with $\pi$ super-Yamanouchi, the \newword{matching of $\rho$ to $\pi$} is constructed as follows: for $i$ from $\ell(w)$ to $1$, set $k = \pi_i$ and for $j$ from $\ell(w)$ to $1$ if $\rho_j$ is already paired then decrement $j$ to $j-1$; otherwise if $\rho_j = k$ then pair $\rho_j$ and $\pi_i$; otherwise if $\rho_j = k-1$ then decrement $k$ to $k-1$ and $j$ to $j-1$; otherwise decrement $j$ to $j-1$. 
  \label{def:match}
\end{definition}

This algorithm is well-defined \cite[Thm.~2.8]{Ass19-inv} and can be used to compute the \newword{inversion number of $\rho$}, a statistic giving the minimum number of Coxeter moves needed to obtain $\rho$ from $\pi$. For example, see Fig.~\ref{fig:rex-inv}.

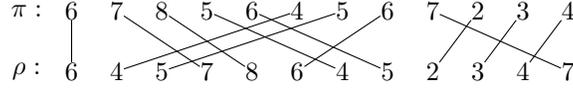
\begin{figure}[ht]
  \begin{center}
    \begin{tikzpicture}[every node/.style={inner sep=0pt},xscale=0.6,yscale=0.4]
      \node at (0,2)  (P0) {$\pi:$};
      \node at (1,2)  (P12) {$6$};
      \node at (2,2)  (P11) {$7$};
      \node at (3,2)  (P10) {$8$};
      \node at (4,2)  (P9)  {$5$};
      \node at (5,2)  (P8)  {$6$};
      \node at (6,2)  (P7)  {$4$};
      \node at (7,2)  (P6)  {$5$};
      \node at (8,2)  (P5)  {$6$};
      \node at (9,2)  (P4)  {$7$};
      \node at (10,2) (P3)  {$2$};
      \node at (11,2) (P2)  {$3$};
      \node at (12,2) (P1)  {$4$};
      \node at (0,0)  (R0) {$\rho:$};
      \node at (1,0)  (R12) {$6$};
      \node at (2,0)  (R11) {$4$};
      \node at (3,0)  (R10) {$5$};
      \node at (4,0)  (R9)  {$7$};
      \node at (5,0)  (R8)  {$8$};
      \node at (6,0)  (R7)  {$6$};
      \node at (7,0)  (R6)  {$4$};
      \node at (8,0)  (R5)  {$5$};
      \node at (9,0)  (R4)  {$2$};
      \node at (10,0) (R3)  {$3$};
      \node at (11,0) (R2)  {$4$};
      \node at (12,0) (R1)  {$7$};
      \draw  (P1) -- (R2) ;
      \draw  (P2) -- (R3) ;
      \draw  (P3) -- (R4) ;
      \draw  (P4) -- (R1) ;
      \draw  (P5) -- (R7) ;
      \draw  (P6) -- (R10) ;
      \draw  (P7) -- (R11);
      \draw  (P8) -- (R5) ;
      \draw  (P9) -- (R6) ;
      \draw (P10) -- (R8) ;
      \draw (P11) -- (R9);
      \draw (P12) -- (R12);
    \end{tikzpicture}
    \caption{\label{fig:rex-inv}The matching of $\rho\in\Red(w)$ to $\pi_w$.}
  \end{center}
\end{figure}

To ensure Lemma~\ref{lem:superY} applies when removing letters of $\rho$, we have the following.

\begin{lemma}
  For $\rho\in\Red(w)$ with $\pi$ super-Yamanouchi, suppose $\pi_i$ matches to $\rho_{p_i}$. Then $\rho_{p_i} \leq \pi_i$, and if $\pi_{i+1} < \pi_{i}$, then $p_{i+1} > p_{i}$.
  \label{lem:match}
\end{lemma}

\begin{proof}
  The algorithm attempts to match $\pi_i$ with a letter of $\rho$ equal to $\pi_i$ but can decrement the target value, proving $\rho_{j_1} \leq \pi_i$. Suppose $\pi_{i+1} < \pi_{i}$, in which case $\pi_i = \pi_{i+1}+1$ by Definition~\ref{def:superY}. We prove the following by induction on the number of Coxeter relations needed to obtain $\rho$ from $\pi$:
  \begin{itemize}
  \item[(i)] if $\rho_{p_{i+1}} < \rho_{p_{i}}$, then $\rho_{p_i} = \rho_{p_{i+1}}+1$;
  \item[(ii)] if $\rho_{p_{i+1}} \geq \rho_{p_{i}}$, then there exist letters $\rho_{p_{i+1}}+1>\ldots>\rho_{p_i}+1$ lying strictly between them in this order.
  \end{itemize}
  For $\pi$, (i) is trivial and (ii) is vacuous, proving the base case. Assuming the claims for $\rho$, suppose $\sigma$ is obtained from $\rho$ by a single commutation or Yang--Baxter relation, and consider the matching of $\pi$ to $\rho$ then to $\sigma$ by this relation. If neither $\pi_{i+1}$ nor $\pi_{i}$ matches to the letters affect by the relation, then the claims holds for $\sigma$. Thus we assume at least one of $\pi_{i+1}$, $\pi_{i}$ matches to an affected letter of $\rho$. 

  \textsc{Case (commutation)}: Suppose $\sigma$ is obtained by a commutation relation from $\rho$, say swapping adjacent letters $j,k$ with $|j-k|>1$ and $j$ to the left of $k$. If only one of $\pi_{i+1}$, $\pi_{i}$ matches to the commuting letters, then the relative positions and values of the images of $\pi_{i+1}$, $\pi_{i}$ remain the same in $\sigma$ as in $\rho$, proving the claims still holds. Thus assume $\pi_{i+1}$, $\pi_{i}$ match to $j,k$, respectively as shown on the left side of Fig.~\ref{fig:relations}. If $j>k$, then $\rho$ contradicts to (ii); if $j=k$, then $\rho$ is not reduced; and if $j<k$, then $|j-k|>1$ contradicts (i). Thus these cases cannot occur.

  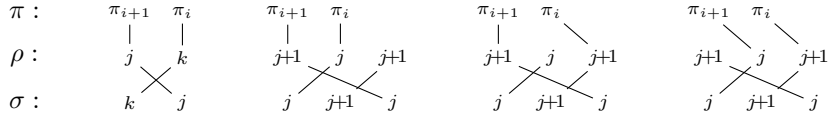
\begin{figure}[ht]
    \begin{center}
      \begin{tikzpicture}[every node/.style={inner sep=2pt},xscale=0.7,yscale=0.6]
        \node at ( 0,3)  (P0) {$\pi:$};
        \node at ( 2,3)  (P1) {$\scriptstyle \pi_{i+1}$};
        \node at ( 3,3)  (P2) {$\scriptstyle \pi_{i}$};
        \node at ( 5,3)  (P3) {$\scriptstyle \pi_{i+1}$};
        \node at ( 6,3)  (P4) {$\scriptstyle \pi_{i}$};
        \node at ( 9,3)  (P6) {$\scriptstyle \pi_{i+1}$};
        \node at (10,3)  (P7) {$\scriptstyle \pi_{i}$};
        \node at (13,3)  (P9) {$\scriptstyle \pi_{i+1}$};
        \node at (14,3)  (PT) {$\scriptstyle \pi_{i}$};
        \node at ( 0,2)  (R0) {$\rho:$};
        \node at ( 2,2)  (R1) {$\scriptstyle j$};
        \node at ( 3,2)  (R2) {$\scriptstyle k$};
        \node at ( 5,2)  (R3) {$\scriptstyle j\!+\!1$};
        \node at ( 6,2)  (R4) {$\scriptstyle j      $};
        \node at ( 7,2)  (R5) {$\scriptstyle j\!+\!1$};
        \node at ( 9,2)  (R6) {$\scriptstyle j\!+\!1$};
        \node at (10,2)  (R7) {$\scriptstyle j      $};
        \node at (11,2)  (R8) {$\scriptstyle j\!+\!1$};
        \node at (13,2)  (R9) {$\scriptstyle j\!+\!1$};
        \node at (14,2)  (RT) {$\scriptstyle j      $};
        \node at (15,2)  (RE) {$\scriptstyle j\!+\!1$};
        \node at ( 0,1)  (S0) {$\sigma:$};
        \node at ( 2,1)  (S1) {$\scriptstyle k$};
        \node at ( 3,1)  (S2) {$\scriptstyle j$};
        \node at ( 5,1)  (S3) {$\scriptstyle j      $};
        \node at ( 6,1)  (S4) {$\scriptstyle j\!+\!1$};
        \node at ( 7,1)  (S5) {$\scriptstyle j      $};
        \node at ( 9,1)  (S6) {$\scriptstyle j      $};
        \node at (10,1)  (S7) {$\scriptstyle j\!+\!1$};
        \node at (11,1)  (S8) {$\scriptstyle j      $};
        \node at (13,1)  (S9) {$\scriptstyle j      $};
        \node at (14,1)  (ST) {$\scriptstyle j\!+\!1$};
        \node at (15,1)  (SE) {$\scriptstyle j      $};
        \draw  (P1) -- (R1) ;
        \draw  (P2) -- (R2) ;
        \draw  (P3) -- (R3) ;
        \draw  (P4) -- (R4) ;
        \draw  (P6) -- (R6) ;
        \draw  (P7) -- (R8) ;
        \draw  (P9) -- (RT) ;
        \draw  (PT) -- (RE) ;
        \draw  (R1) -- (S2) ;
        \draw  (R2) -- (S1) ;
        \draw  (R3) -- (S5) ;
        \draw  (R4) -- (S3) ;
        \draw  (R5) -- (S4) ;
        \draw  (R6) -- (S8) ;
        \draw  (R7) -- (S6) ;
        \draw  (R8) -- (S7) ;
        \draw  (R9) -- (SE) ;
        \draw  (RT) -- (S9) ;
        \draw  (RE) -- (ST) ;
      \end{tikzpicture}
      \caption{\label{fig:relations}The four potential cases of $\pi$ matching to $\rho$ with an additional Coxeter relation increasing the crossing number.}
    \end{center}
  \end{figure}

  \textsc{Case (Yang--Baxter)}: Suppose $\sigma$ is obtained by a commutation relation from $\rho$, say braiding $j+1,j,j+1$ to $j,j+1,j$. If only one of $\pi_{i+1}$, $\pi_{i}$ matches to one of the braiding letters, then the relative positions of the images of $\pi_{i+1}$, $\pi_{i}$ remain the same in $\sigma$ as in $\rho$, proving the first claim. If $\pi_{i+1}$ matches to an unaffected letter, then either $\pi_{i}$ matches to the same value with the same set of larger values between them, or $\pi_{i}$ matches to $j$ with $j+1$ immediately to its left, proving the claims for $\sigma$. Similarly, if $\pi_{i}$ matches to an unaffected letter, then either $\pi_{i+1}$ matches to the same value or to $j$, again proving the claims for $\sigma$. Finally, we have left the three cases where both $\pi_{i+1}$, $\pi_{i}$ match to affected letters of $\rho$; see Fig.~\ref{fig:relations}. The middle two cases contradict claim (ii), and so cannot occur, and the rightmost case maintains relative positions and values for the images in $\sigma$, proving the claims.  
\end{proof}

\begin{theorem}
  For $(\rho,\alpha)$ a compatible pair with $\rho\in\Red(w)$, $\D(\rho,\alpha)$ is a well-defined Kohnert diagram for $\D(w)$ and 
  \begin{equation}
    x_{\alpha_1} \cdots x_{\alpha_{n}} = x_1^{\wt(\D(\rho,\alpha))_1} \cdots x_n^{\wt(\D(\rho,\alpha))_n}.  \vspace{-\baselineskip}
    \label{e:forward-weight}
  \end{equation}      
  \label{thm:forward}
\end{theorem}

\begin{proof}
  Consider the case $\ell(w)=1$, where we have $\Red(w)=\{(r)\}$. Here $\D(w)$ is the diagram with a single cell in row $r$, column $r$. The diagram $\D((r),\alpha)$ is a single cell in row $\alpha_1$, column $r$ where $\alpha = (\alpha_1)$. Any compatible word $\alpha = (\alpha_1)$ has, by definition, $\alpha_1 \leq r$, and so $\D((r),\alpha)$ is obtained by exactly $r - \alpha_1\geq 0$ Kohnert moves on $\D(w)$. Thus, $\D((r),\alpha) \in \KD(w)$, so we proceed by induction on $\ell(w)$.

  Suppose $\rho \in \Red(w)$ with $\ell=\ell(w)>1$, and take $\alpha$ any $\rho$-compatible word. Let $\widehat{\rho} = (\rho_{\ell},\cdots,\rho_2)$, $\widehat{\alpha} = (\alpha_{\ell},\cdots,\alpha_2)$, and $\widehat{w} = w s_{\rho_1}$. Then $\widehat{\rho}\in\Red(\widehat{w})$ has compatible word $\widehat{\alpha}$ and $\ell(\widehat{w}) = \ell(w)-1$. By induction, $\D(\widehat{\rho},\widehat{\alpha}) \in \KD(\widehat{w})$ and \eqref{e:forward-weight} holds for $\widehat{\rho},\widehat{\alpha}$. Since every pair of adjacent columns of $\D(\widehat{w})$ is ordered by containment, $\mathfrak{s}_{\rho_1}\D(\widehat{\rho},\widehat{\alpha})$ is a well-defined Kohnert diagram for $\mathfrak{s}_{\rho_1}\D(\widehat{w})$. The $\rho$-compatible conditions state $\alpha_{i+1} \geq \alpha_{i}$ with $\alpha_{i+1} > \alpha_{i}$ whenever $\rho_{i+1} > \rho_i$. Thus if $\rho_{2} > \rho_1$, then the new cell $x$ is placed into a strictly lower row, and if $\rho_{2} < \rho_1$, then $x$ is placed into a strictly rightward column. Therefore $\D(\rho,\alpha)$ is well-defined and \eqref{e:forward-weight} holds for $\D(\rho,\alpha)$. It remains only to show $\D(\rho,\alpha)\in\KD(w)$.
  
  Suppose $\rho_1$ matches to some letter $\pi_{r}$ for $\pi = \pi_w$ super-Yamanouchi. By Lemma~\ref{lem:match}, either $r=1$ or $\pi_r > \pi_{r-1}$, and either way $\pi_r$ lies at the end of its row in $\D(w)$. Thus by Lemma~\ref{lem:superY}, $\widehat{\pi} = \pi_{\ell}\cdots\pi_{i+1}\pi_{i-1}\cdots\pi_1$ is super-Yamanouchi for $\widehat{w}$ and $\D(\widehat{w}) = \mathfrak{s}_{\rho_1} \left( \D(w) \setminus \{(\pi_w)_i\} \right)$. Thus we can insert a cell in row $\pi_r$, column $\rho_1+1$ of $\D(\widehat{w})$ that we push, via Kohnert moves, down to row $\alpha_1$, and then perform the remaining Kohnert moves to obtain $\mathfrak{s}_{\rho_1}\D(\widehat{\rho},\widehat{\alpha})$ union a cell in row $\alpha_1$, column $\rho_1$. Therefore $\D(\rho,\alpha)$ is indeed a Kohnert diagram for $w$.
\end{proof}

To show $\T$ is well-defined, we give the analog of Lemma~\ref{lem:match} for Kohnert diagrams.

\begin{lemma}
  Let $T\in\KD(w)$ and suppose $x$ is the lowest then rightmost cell of $T$. If $x$ lies in column $c$, then some row of $\D(w)$ has its rightmost cell in column $c$.
  \label{lem:end}
\end{lemma}

\begin{proof}
  Let $r_1 < \cdots < r_k$ be the rows of $\D(w)$ with a cell in column $c$. If some row $r_i$ has its rightmost cell in column $c$, then the lemma is proved. Otherwise, we may take $c'>c$ the smallest column index such that some row $r_i$ has a cell in column $c'$. We claim all rows $r_i$ have a cell in column $c'$. If not, then let $i$ be the smallest row index such that row $r_i$ has no cell in column $c'$. We may describe $\D(w)$ by placing bullets in row $i$, column  $w_i$ and killing all cells above and to the right, as illustrated in Fig.~\ref{fig:Rothe}. If $i=1$, then since $r_1$ has a cell in column $c<c'$ and no cell in column $c'$, there is a bullet in column $c'$ below row $r_1$ ensuring no row $r>r_1$ can have a cell in column $c'$, contradicting the choice of $c'$. If $i>1$, then there is a bullet column $c'$ in some row $r$ with $r_{i-1} > r \geq r_i$. Since both rows $r_{i-1},r_i$ have a cell in column $c$, the bullet in column $c$ lies above row $r_i>r$, and so row $r$ must have a cell in column $c$ as well, in which case $r=r_i$. However, by choice of $c'$ and $r_i$, row $r_i$ must have a cell strictly to the right of $c'$, a contradiction. Therefore all rows $r_i$ have cells in both columns $c<c'$, and no cells between these. To arrive at a contradiction, we claim for any $S\in\KD(w)$ and any $r\geq 1$, we have
  \begin{equation}
    \#\{x\in S \mid \mathrm{row}(x) \leq r, \ \mathrm{col}(x) = c \} \leq
    \#\{x\in S \mid \mathrm{row}(x) \leq r, \ \mathrm{col}(x) = c' \} .
    \label{e:2.2}
  \end{equation}
  As the row of $x$ in $T$ violates this inequality, we have a contradiction, thereby proving the lemma. Suppose $S\in\KD(w)$ satisfies \eqref{e:2.2}, but applying a Kohnert move to a cell $y$ in $S$ results in a violation. Since $y$ moved down, we must have $y$ in column $c$, say going from row $s$ down to the nearest empty position in row $r<s$ causing \eqref{e:2.2} to fail for row $r$. Since \eqref{e:2.2} holds for $S$, we must have equality at row $r$ for $S$. Since pushing $y$ is a Kohnert move, every row above $r$ and weakly below $s$ has a cell in column $c$. For \eqref{e:2.2} to hold for $S$ for rows up to and including $s$, the same must be true for column $c'$ as well. However, this means there is a cell in row $s$, column $c'$, a contradiction to pushing $y$ being a Kohnert move. Therefore Kohnert moves preserve \eqref{e:2.2}, and the theorem follows.
\end{proof}

\begin{theorem}
  For $T$ a Kohnert diagram for $\D(w)$, $\W(T)$ is a well-defined reduced word for $w$ and $\alpha(T) = (n^{\wt(T)_n},\ldots,1^{\wt(T)_1})$ is $\W(T)$-compatible. 
  \label{thm:backward}
\end{theorem}

\begin{proof}
  Consider the case $\ell(w)=1$, so that $\Red(w)=\{(r)\}$ and $\D(w)$ has a single cell in row $r$, column $r$. Every Kohnert diagram $T$ has a single cell in column $r$ weakly below row $r$, and so $\W(T) = (r)\in\Red(w)$ and $\alpha(T)=(\alpha_1)$ where $\alpha_1\leq r$ is the row of the cell. Thus we may proceed by induction on $\ell(w)$.

  Suppose $T\in\KD(w)$ with $\ell=\ell(w)>1$, and let $x$ be the lowest then rightmost cell of $T$, say in column $c$. By Lemma~\ref{lem:end}, some row of $\D(w)$ ends in column $c$. In this case, there exists an index $r$ such that for $\pi=\pi_w$ super-Yamanouchi, $\pi_r$ lies at the end of its row and in column $c$ and no lower row has its rightmost cell in column $c$. By Lemma~\ref{lem:superY}, deleting this letter from $\pi$ yields the super-Yamanouchi word for $\widehat{w} = w s_{\rho_c}$ and $\D(\widehat{w}) = \mathfrak{s}_{c} \left( \D(w) \setminus \{\pi_r\} \right)$. Thus $T\setminus x \in \KD(\D(w) \setminus \{\pi_r\})$, and so by Lemma~\ref{lem:vexillary}, we have $\mathfrak{s}_c(T\setminus x) \in \KD(\mathfrak{s}_{c} \left( \D(w) \setminus \{\pi_r\} \right) = \KD(\widehat{w})$. Since $\ell(\widehat{w}) = \ell(w)-1$, by induction $\W(\mathfrak{s}_c(T\setminus x)) \in \Red(\widehat{w})$, and so $\W(T) = \W(\mathfrak{s}_c(T\setminus x)) c \in \Red(w)$. 

  From \eqref{e:rothe}, the lowest cell of $\D(w)$ in column $c$ is at most the smallest index $i$ such that $w_i > c$. Therefore the lowest cell in column $c$ lies weakly below row $c$ since otherwise the first $c$ letters of $w$ must be strictly less than $c$, an impossibility. Extrapolating this, the $k$th lowest cell in column $c$ must lie weakly below row $c+k-1$ for all $c,k$. As Kohnert moves push cells down within their columns, this implies the $k$th lowest cell in column $c$ of $T\in\KD(w)$ must also lie weakly below row $c+k-1$. Thus for $x$ the lowest then rightmost cell of $T$ and $c$ its column index, $\alpha(T)_1$ is the row of $x$, and so $\alpha(T)_1 \leq c$. Removing $x$ ensures $\mathfrak{s}_c \cdot (T\setminus x)$ also has the property that the $k$th lowest cell in column $c'$ lies weakly below row $c'+k$ for all $c'$, and we may thus proceed as before. At each step, the recorded letter is at least as large as the row index of the removed cell, showing $\alpha(T)_i \leq \W(T)_i$. Moreover, if $\W(T)_{i+1} > \W(T)_{i}$, then the corresponding removed cells $x_{i+1},x_{i}$ must have $x_{i}$ weakly left of $x_{i+1}$. Since we remove the lowest then rightmost cell, this ensures $x_{i+1}$ lies strictly above $x_i$, and so $\alpha(T)_{i+1} > \alpha(T)_i$. Thus $\alpha(T)$ is $\W(T)$-compatible.
\end{proof}

The maps $\D$ and $\T$ are clearly inverse to one another, and so Theorems~\ref{thm:forward} and \ref{thm:backward} give a bijective proof of Kohnert's rule for Schubert polynomials.

\vspace{\baselineskip}

\begin{center}
{\sc Acknowledgments}
\end{center}

Special thanks to Adelie Borman and the Venice Math Circle for verifying the simplicity of and computing examples for the bijections described in this paper. Thanks as well to Peter Kagey for helpful suggestions that clarified exposition and to Shiliang Gao for careful reading that caught misstatements in earlier drafts.

%
%

\bibliographystyle{plain} 
\bibliography{kohnertsrule}

\end{document}